\documentclass[12pt]{amsart}
\usepackage{amsmath}
\usepackage{amsxtra}
\usepackage{amscd}
\usepackage{amsthm}
\usepackage{amsfonts}
\usepackage{amssymb}
\usepackage{eucal}

\input prepictex
\input pictex
\input postpictex
\usepackage{mathptm}

\usepackage{verbatim}

\def\d{\delta}

\def\a{\alpha}
\def\l{\lambda}
\def\e{\epsilon}

\def\D{\Delta}
\def\G{\Gamma}
\def\L{\Lambda}
\def\I{{\bf I}}

\def\K{{\mathbb K}}
\def\C{\mathbb{C}}
\def\Z{\mathbb{Z}}

\def\gl{{\mathfrak{gl}_{n}}}
\def\Uq{{\rm{U}}_q{(\frak{gl}_{n})}}
\def\UW{({{U}_q}\!(\gl)^0_{ev})^W}
\def\sm{\!\setminus\!}

\def\gs{\geqslant}

\def\HH{{\mathcal H}}

\def\ZZ{{\mathcal{Z}}}

\def\cl{\centerline}
\def\QED{\hfill$\Box$\par}

\def\rar{\rightarrow}

\def\vs{\vspace*}

\numberwithin{equation}{section}

\newtheorem{theo}{Theorem}[section]
\newtheorem{lemm}{Lemma}[section]

\newtheorem{coro}{Corollary}[section]

\begin{document}

\title[The quantum Casimir operators of $\Uq$]
{The quantum Casimir operators of $\Uq$ and their
eigenvalues}

\author{Junbo Li}
\address{School of Mathematics and Statistics, University of Sydney, NSW
2006, Australia}

\begin{abstract}
We show that the quantum Casimir operators of the quantum linear
group constructed in early work of Bracken, Gould and Zhang together
with one extra central element generate the entire center of $\Uq$.
As a by product of the proof, we obtain intriguing new formulae for
eigenvalues of these quantum Casimir operators, which are expressed
in terms of the characters of a class of finite dimensional
irreducible representations of the classical general linear algebra.

\medskip

\noindent{\it Keywords:}\ \ Quantum groups, Harish-Chandra isomorphism, Center

\medskip

\noindent{\it{MR(2010) Subject Classification}: 17B10, 17B37}\vs{10pt}
\end{abstract}
\maketitle

\section{Introduction}

Quantum groups originated from the theory of soluble models of the
Yang-Baxter type in the middle of 80s. They have played important
roles in various branches of mathematics and physics, most notably
in two-dimensional soluble models in statistical mechanics and knot
theory. The study of their structure and representation theory has
been the focus of research in representation theory and continues to
attract much attention. In particular, the discovery of the crystal
basis and canonical basis \cite{K, L} is one of the most important
achievement in representation theory in recent years.

In the early 90s, a set of central elements of the quantum group
$\Uq$ was constructed in \cite{GZB, ZGB}. We shall refer to them as
the quantum Casimir operators of $\Uq$. The reason for this
terminology is the fact that these central elements of $\Uq$ are the
quantum analogues of the familiar Casimir operators of the universal
enveloping algebra ${\rm U}(\mathfrak{gl_n})$ of $\mathfrak{gl_n}$
given by $ \sum_{i_1=1}^n\cdots \sum_{i_k=1}^n E_{i_1 i_2} E_{i_2
i_3}\cdots E_{i_{k-1} i_k} E_{i_k i_1}$ $(k=1,2,\cdots)$, where
$E_{i j}$ are the images of the matrix units under the canonical
embedding of $\mathfrak{gl_n}$ in ${\rm U}(\mathfrak{gl_n})$. These
Casimir operators of ${\rm U}(\mathfrak{gl_n})$ play an important
role in the Interacting Boson Model in nuclear physics. Their
quantum analogues have also been applied in a similar way.

One obvious question was whether the quantum Casimir operators
\cite{GZB, ZGB} of $\Uq$  (supplemented with the obvious central
element $c$ given by \eqref{c-inv}) generated the entire center of
$\Uq$. The general expectation was that the answer should be
affirmative, but no proof was ever given as far as we know. The main
purpose of this paper is to give a rigorous proof. The result is
described in Theorem \ref{MT}.

The proof of Theorem \ref{MT} requires us to analyze the eigenvalues
of the quantum Casimir operators in highest weight representations
of $\Uq$. A formula for the eigenvalues was obtained in \cite{LZ}
(in fact \cite{LZ} treated ${\rm U}_q(\mathfrak{gl_{m|n}})$, which
included $\Uq$ as a special case). We cast the formula into a form
readily usable for our purpose. This new formula is expressed in
terms of the characters of a class of finite dimensional irreducible
representations of the classical general linear algebra. This result
is rather intriguing, and we believe that it is interesting in its
own right.

We should point out that the structure of the center of a quantum
group is much studied \cite{J, T} at an abstract level. In
particular, a quantum analogue of the celebrated Harish-Chandra
homomorphism in semi-simple Lie algebras has been established for
quantum groups at generic $q$. In the case of $\Uq$, a set of
generators different from the quantum Casimir operators of
\cite{GZB} was constructed in \cite{HM}.

\section{The quantum general linear group}
\subsection{The quantum general linear group}

This section provides some basic materials on the general linear
algebra $\gl$ and its quantum group $\Uq$. Let $\e_i$,  with $i\in\I
= \{1,2,\cdots,n\}$, be a basis of an Euclidean space with the inner
product $(\e_i, \e_j)=\delta_{i j}$. Set
$\rho=\frac{1}{2}\sum_{i=1}^n(n-2i+1)\e_i$.

The quantized universal enveloping algebra $\Uq$ of the general
linear Lie algebra $\gl$ is a unital associative algebra over
$\K:=\C(q)$ generated by $K_i^{\pm1}$, $E_{i'}$, $F_{i'}$ $(i\in \I,
i'\in \I':=\I\sm\{n\})$, subject to the following relations:
\begin{eqnarray*}
&&K_iK_i^{-1}=1,\ \ \ K_i^{\pm1} K_j^{\pm1}=K_j^{\pm1}K_i^{\pm1},\ \ \
K_iE_{j} K_i^{-1}=q^{(\epsilon_j-\e_{j+1}, \epsilon_i)}E_{j},\\
&&K_iF_{j} K_i^{-1}=q^{-(\epsilon_j-\e_{j+1}, \epsilon_i)}F_{j}, \ \ \
E_{r}E_{s}=E_{s}E_{r},\ \ F_{r}F_{s}=F_{s}F_{r},\quad |r-s|\gs2,\\
&&E_iF_j-F_jE_i=\d_{i,j}\frac{K_iK^{-1}_{i+1}-K_i^{-1}K_{i+1}}{q-q^{-1}},
\ \ \ S^{(+)}_{i, i\pm1}=S^{(-)}_{i, i\pm1}=0,
\end{eqnarray*}
and
\begin{eqnarray*}
S^{(+)}_{i,i\pm 1}\!\!\!&=&\!\!\!(E_{i})^2E_{i\pm 1}-(q+q^{-1})E_{i}E_{i\pm1}E_{i}+E_{i\pm1}(E_{i})^2,\\
S^{(-)}_{i, i\pm 1}\!\!\!&=&\!\!\!(F_{i})^2F_{i\pm1}-(q+q^{-1})F_{i}F_{i\pm 1}F_{i}+F_{i\pm 1}(F_{i})^2.
\end{eqnarray*}

As we know, $\Uq$ possesses the structure of a Hopf algebra with the
co-multiplication $\D$,  co-unit $\e$ and antipode $S$ respectively
given by
\begin{eqnarray*}
&&\Delta(E_{i})=E_{i}\otimes K_iK_{i+1}^{-1}+1\otimes E_{i},\ \ \ \Delta(F_{i})=F_{i}\otimes 1+K_i^{-1}K_{i+1}\otimes F_{i},\\
&&\Delta(K^{\pm1}_i)=K^{\pm1}_i\otimes K^{\pm1}_i,\ \ \ \epsilon(E_{i})=\epsilon(F_{i})=0,\ \ \epsilon(K^{\pm1}_i)=1,\\
&&S(E_{i})=-E_{i}K_i^{-1}K_{i+1},\ \ \ S(F_{i})=-K_iK_{i+1}^{-1}F_{i},\ \ \ S(K^{\pm1}_i) =K^{\mp1}_i.
\end{eqnarray*}

The natural module $V$ for $\Uq$ has the standard basis
$\{v_i\,|\,i\in\I\}$ such that
\begin{eqnarray*}
E_iv_j=\d_{j,i+1}v_i,\ \ \ F_{i}v_j=\d_{j,i}v_{i+1},\ \ \ K_iv_j=\big(1+(q-1)\d_{i,j}\big)v_j.
\end{eqnarray*}
Denote $\pi$ the $\Uq$-representation relative to this basis, then
$\pi(E_{i})=E_{i, i+1}$, $\pi(F_{i})=E_{i+1, i}$ and $\pi(K_{i})=I +
(q-1)E_{i i}$, where $E_{i j}$ are the matrices $(E_{i j})_{r
s}=\d_{i r}\d_{j s}$.

We now turn to the description of the center $\ZZ$ of the quantum
$\Uq$. Let $\Uq^-$, $\Uq^+$ and $\Uq^0$ the subalgebras of $\Uq$
generated by $F_{i'}$ ($i'\in I'$), $E_{i'}$ ($i'\in I'$) and
$K_i^\pm$ ($i\in I$) respectively. Any element $z\in\ZZ$ can be
written as
\[
z = z^{(0)}+\mbox{$\sum\limits_s u_s^{(-)}$}u_s^{(0)} u_s^{(+)},
\]
where $z^{(0)}, u_s^{(0)}\in \Uq^0$, $u_s^{(+)}\in\Uq^+$ and
$u_s^{(-)}\in \Uq^-$. The quantum Harish-Chandra homomorphism is an
algebra homomorphism $\varphi: \ZZ \longrightarrow \Uq^0$ such that
$z\mapsto z^{(0)}$.

The dot action of the Weyl group $W$ of $\gl$ on $\Uq^0$ is given by
permutations of the elements $q^{-i}K_i$  ($i\in{\bf I}$). Define
\begin{eqnarray}
L_i = q^{(\e_i,2\rho)}(K_i)^2, \quad i\in{\bf I}.
\end{eqnarray}
Then $L_i$ are permuted by the Weyl group. Let $\Uq^0_{ev}$ be  the
subalgebra spanned by the elements $\Pi_{i=1}^nL_i^{l_i}$ for
$l_i\in\Z$. We denote by $\UW$ the $W$-invariant subalgebra of
$\Uq^0_{ev}$. By using the quantum Harish-Chandra isomorphism for
$U_q(\frak{sl}_n)$ (see, e.g., \cite{J, T}), one can prove the
following result.
\begin{lemm}\label{lamma2--1}
The Harish-Chandra homomorphism is an algebra isomorphism between
the center $\ZZ$ of $\Uq$ and the subalgebra of $\Uq^0$ generated by
the elements of $\UW$ together with $c$, where
\begin{eqnarray}\label{c-inv}
c=K^{-1}_1 K^{-1}_2 \cdots K_n^{-1}.
\end{eqnarray}
\end{lemm}
Note that $c$ is obviously $W$-invariant, and $c^{\pm 2}\in \UW$. An
equivalent description of this lemma can be found in \cite{HM}. In
this paper, a set of generators for $\ZZ$ were also given, which are
different from the quantum Casimirs operators of \cite{GZB, ZGB}.

\subsection{Quantum Casimir operators of $\Uq$}
The quantum Casimir operators of $\Uq$ are the main objects for
study in this paper, which we now briefly describe. As is well
known, in the quantum group setting, we have neither a good quantum
analogue of tensors nor a procedure for ``contracting tensors"
(however, see \cite{qtensor1, qtensor2}). Thus it is much harder to
explicitly construct central elements for quantum groups. The
construction of \cite{GZB, ZGB} was actually quite involved: it had
to invoke the theory of \cite{knots} and also made use the universal
$R$-matrix of $\Uq$. Thus for the sake of completeness and also
clearness we briefly explain the construction.

It is well-known that $\Uq$ is a quasi-triangular Hopf algebra,
i.e., there exists an invertible element $R\in\Uq\otimes\Uq$ which
is called the universal $R$-matrix of $\Uq$, such that
\begin{eqnarray*}
\begin{aligned}
R\D(x)=\D'(x)R, \quad \forall x\in\Uq, \\
(\D\otimes id)R=R_{13}R_{23},\quad (id\otimes\D)R=R_{13}R_{12},
\end{aligned}
\end{eqnarray*}
where $\Delta'$ is the opposite co-multiplication. Explicitly,
$\Delta'=T\circ \Delta$ where $T: \Uq\otimes\Uq\rar\Uq\otimes\Uq$ is
the linear map defined for any $x,y\in\Uq$ by $T(x\otimes
y)=y\otimes x$.  The $R$-matrix satisfies the celebrated Yang-Baxter
equation
\[ R_{12}R_{13}R_{23}=R_{23}R_{13}R_{12}.\]
Denote $R^T=T(R)$. Then $R^TR\D(x)=\D(x)R^TR$, $\forall$ $x\in\Uq$

For any $\a\in\Z$, Let $\Z_{\geqslant
\a}:=\{m\,|\,m\in\Z,\,m\geqslant\a\}$.  The following quantum
Casimir operators for $\Uq$ were constructed in \cite{GZB, ZGB}:
\begin{eqnarray*}
C_{n,k}={\rm Tr}_\pi\left(1\otimes
q^{\pi(2h\rho)}\left(\frac{\G-1\otimes1}{q-q^{-1}}\right)^k\right),\
\ k\in\Z_{\geqslant0},
\end{eqnarray*}
where ${\rm Tr}_\pi$ represents the trace taken over $\pi$, and
\begin{eqnarray*}
\G=(id\otimes\pi)R^TR,\ \ \
q^{\pi(2h\rho)}=\mbox{$\prod\limits_{i=1}^n$}q^{(n-2i+1)E_{i,i}}.
\end{eqnarray*}

Let $L_\L$ be a finite dimensional irreducible $\Uq$-module with
highest weight $\L\in\HH^*$. Then each $C_{n,k}$ acts on $L_\L$ by a
scalar, which is given by the following formula \cite{LZ}:
\begin{eqnarray*}
\chi_\L(C_{n,k})\!\!\!&=&\!\!\!(q-q^{-1})^{-k}\mbox{$\sum\limits_{i=1}^{n}$}
\big(q^{(\e_i,2\L+2\rho+\e_i)-C(\L_0)}-1\big)^k\\
\!\!\!&&\!\!\!\times
q^{C(\L_0)-(\e_i,\e_i)}\mbox{$\prod\limits_{j\neq i}^{n}$}
\frac{q^{(\e_i,2\L+2\rho+\e_i)}-q^{(\e_j,2\L+2\rho-\e_j)}}
{q^{(\e_i,2\L+2\rho+\e_i)}-q^{(\e_j,2\L+2\rho+\e_j)}},
\end{eqnarray*}
where $\L_0=\epsilon_1$.

Since $\varphi(z)$ belongs to $\Uq^0$ for every $z\in\ZZ$, the right
hand side of the above formula must be a polynomial in
$q^{(\e_i,2\L+2\rho)}$. Note that when applying $L_i$ to the highest
weight vector $v_{\L}$ of $L_\L$, we have
$L_iv_{\L}=q^{2(\e_i,\L+\rho)}v_\L$. Then it follows that
\begin{eqnarray}\label{formula}
C^0_{n,k}:=\varphi(C_{n,k})
=\mbox{$\sum\limits_{i=1}^{n}$}\big(\frac{q^{1-n}L_i-1}{q-q^{-1}}\big)^k
\mbox{$\prod\limits_{j\neq i}^{n}$}\frac{qL_i-q^{-1}L_j}{L_i-L_j}.
\end{eqnarray}

\section{Analysis of $G_{n,k}$}
Let us analyze the formula for $C^0_{n,k}$ to put it into a form
which will be readily usable for the proof of Theorem \ref{MT}.
Denote
\begin{eqnarray*}
G_{n,k}=\mbox{$\sum\limits_{i=1}^{n}$}{L_i}^kP_{n,i}\ \
\mbox{where}\ \ P_{n,i}=\mbox{$\prod\limits_{j\neq
i}^{n}$}\frac{qL_i-q^{-1}L_j}{L_i-L_j}.
\end{eqnarray*}
Then for any
$n\in\Z_{\geqslant 2}$ and $k\in\Z_{\geqslant 1}$, one can rewrite
$C^0_{n,k}$ as
\[
C^0_{n,k} = \frac{1}{(q^{-1}-q)^k}\sum_{j=0}^{k} \left(\begin{array}{ll}
k\\j
\end{array}\right)
(-q^{1-n})^j
G_{n,j}.
\]
We want to prove that $G_{n,k}\,(k=0,1,\cdots,n)$ are polynomials in
$L_i$.

\begin{lemm}\label{CM1}
For any $n\geqslant2$, the following two identities hold:
\begin{eqnarray}
G_{n,0}\!\!\!&=&\!\!\!q^{n-1}+q^{n-3}+\cdots+q^{3-n}+q^{1-n},\label{100312a01}\\
G_{n,1}\!\!\!&=&\!\!\!q^{n-1}(L_1+L_2+\cdots+L_n).\label{100312a02}
\end{eqnarray}
\end{lemm}
{\bf Proof}\ \ For $i\neq n,n-1$, one has
\begin{eqnarray*}
P_{n,i}\!\!\!&=&\!\!\!P_{n-2,i}\cdot\frac{qL_i-q^{-1}L_{n-1}}{L_i-L_{n-1}}\cdot\frac{qL_i-q^{-1}L_n}{L_i-L_n}\\
\!\!\!&=&\!\!\!P_{n-2,i}\cdot\frac{qL_i-q^{-1}L_{n-1}}{L_i-L_{n-1}}\cdot\frac{qL_i-q^{-1}L_n}{L_{n-1}-L_{n}}
-P_{n-2,i}\cdot\frac{qL_i-q^{-1}L_n}{L_i-L_{n}}\cdot\frac{qL_i-q^{-1}L_{n-1}}{L_{n-1}-L_{n}}.
\end{eqnarray*}
For any $m,n\in\Z_{\geqslant2}$ and $i\in\Z_{\geqslant0}$,  denote
$P_{m,i,n}=P_{m,i}\frac{qL_i-q^{-1}L_{n}}{L_i-L_{n}}$. Then
$P_{n,i}$ can be rewritten as
\begin{eqnarray}\label{0126m01}
P_{n,i}\!\!\!&=&\!\!\!P_{n-2,i,n-1}
\cdot\frac{qL_i-q^{-1}L_n}{L_{n-1}-L_{n}}
-P_{n-2,i,n}\cdot\frac{qL_i-q^{-1}L_{n-1}}{L_{n-1}-L_{n}}.
\end{eqnarray}
For convenience, denote
$G_{n,k,\widehat{j}}=\mbox{$\sum\limits_{j\neq
i=1}^{n}$}{L_i}^kP_{n,i}$. Then we can rewrite $G_{n,0}$ as follows:
\begin{eqnarray}
G_{n,0}\!\!\!&=&\!\!\!\mbox{$\sum\limits_{i=1}^{n-2}$}P_{n-2,i,n-1}\cdot\frac{qL_i-q^{-1}L_n}{L_{n-1}-L_{n}}+P_{n,n-1}
-\mbox{$\sum\limits_{i=1}^{n-2}$}P_{n-2,i,n}\cdot\frac{qL_i-q^{-1}L_{n-1}}{L_{n-1}-L_{n}}+P_{n,n}\nonumber\\
\!\!\!&=&\!\!\!\frac{qG_{n-1,1}-q^{-1}L_nG_{n-1,0}}{L_{n-1}-L_{n}}
-\frac{qG_{n,1,\widehat{n-1}}-q^{-1}L_{n-1}G_{n,0,\widehat{n-1}}}{L_{n-1}-L_{n}}.\label{0126m02}
\end{eqnarray}
Meanwhile, for any $n\in\Z_{\geqslant2}$ and $k\in\Z_{\geqslant1}$,
we have the following computations:
\begin{eqnarray*}
G_{n,k}-L_nG_{n,k-1}\!\!\!&=&\!\!\!\mbox{$\sum\limits_{i=1}^n$}L_i^{k-1}(L_i-L_n)P_{n,i}\\
\!\!\!&=&\!\!\!\mbox{$\sum\limits_{i=1}^{n-1}$}L_i^{k-1}P_{n-1,i}(qL_i-q^{-1}L_n)\\
\!\!\!&=&\!\!\!qG_{n-1,k}-q^{-1}L_nG_{n-1,k-1}.
\end{eqnarray*}
Thus
\begin{eqnarray}\label{0126a01}
G_{n,k}\!\!\!&=&\!\!\!L_nG_{n,k-1}+qG_{n-1,k}-q^{-1}L_nG_{n-1,k-1},\
\ \forall\,\,n\in\Z_{\geqslant2},\,k\in\Z_{\geqslant1}.
\end{eqnarray}
It is easy to see that both (\ref{100312a01}) and (\ref{100312a02})
hold for the cases $n=2$ and $n=3$. Conbining the identities
(\ref{0126m02}) and (\ref{0126a01}), using induction on $n$ in
(\ref{100312a01}) and (\ref{100312a02}), one can get the formulae
for $G_{n,0}$ and $G_{n,1}$. Hence we complete the proof of this
lemma.\QED

We introduce a generating function
$S_n(t)=\mbox{$\sum\limits_{k=0}^{\infty}$}t^kG_{n,k}$. Then
\begin{eqnarray*}
S_n(t)-G_{n,0}-tL_nS_n(t)=qS_{n-1}(t)-q^{-1}tL_nS_{n-1}(t)-qG_{n-1,0},
\end{eqnarray*}
which implies
\begin{eqnarray*}
S_n(t)\!\!\!&=&\!\!\!\frac{q-q^{-1}tL_n}{1-tL_n}S_{n-1}(t)+\frac{G_{n,0}-qG_{n-1,0}}{1-tL_n}\\
\!\!\!&=&\!\!\!\frac{q-q^{-1}}{1-tL_n}S_{n-1}(t)+q^{-1}S_{n-1}(t)+\frac{q^{-(n-1)}}{1-tL_n}.
\end{eqnarray*}
Thus
\begin{eqnarray}\label{100303a01}
G_{n,k}=qG_{n-1,k}+(q-q^{-1})\mbox{$\sum\limits_{i=0}^{k-1}$}L_n^{k-i}G_{n-1,i}+q^{-(n-1)}L_{n}^k.
\end{eqnarray}

Note that it is by no means obvious from formula \eqref{formula}
itself that its right hand side is a polynomial in $L_i$'s. To put
our mind at peace, we observe the following result.

\begin{lemm}\label{CM2}
For any $n\in\Z_{\geqslant2}$ and $k\in\Z_{\geqslant0}$, $G_{n,k}$
is a polynomial in $L_i$'s.
\end{lemm}
\begin{proof} This lemma follows from Lemma \ref{CM1} and the identities
(\ref{0126a01}) and (\ref{100303a01}).
\end{proof}

Form the formulae proved in Lemma \ref{CM1},  we see that if $L_i$
is replaced by $e^{\e_i}$ for $i=1,2,\cdots,n$, then the
coefficient of $q^{n-1}$ in $G_{n,1}$ corresponds to the character
of the basic irreducible representation of $\Uq$. A natural problem
is to understand all the $G_{n,k}$ in similar terms, and we address
this problem now.

Let us introduce a set of elements of $\l^i_k$ in
$\sum_{i=1}^n\Z_+\epsilon_i$, which we write in terms of their
coordinates relative to the basis $\epsilon_i$. We let
$\l^1_k=(k,0,0,\cdots,0)$, $\l^2_k=(k-1,1,0,\cdots,0)$, $\cdots$,
$\l^k_k=(1,\cdots,1,0,\cdots,0)$ in the case $k<n$. We also set
$\l^1_k=(k,0,0,\cdots,0)$, $\l^2_k=(k-1,1,0,\cdots,0)$, $\cdots$,
$\l^k_k=(k-n+1,1,\cdots,1)$ in the case
$k\geqslant n$. Note that these weights respectively correspond to
the following Young diagrams:\vs{-40pt}

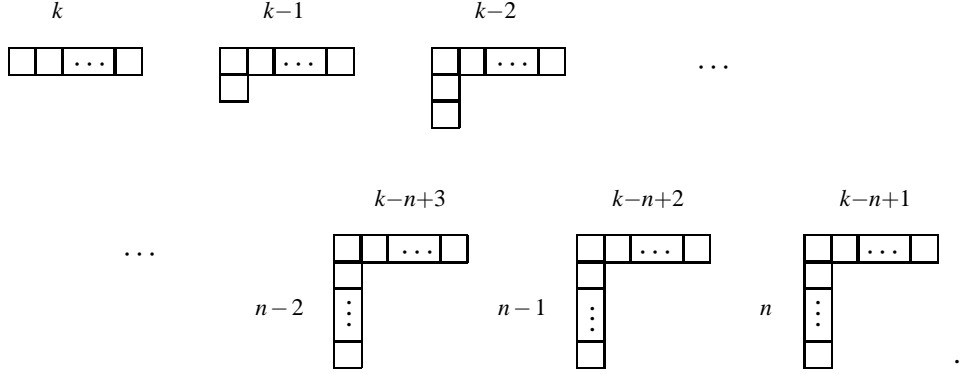
\cl{\begin{picture}(60,160)
\put(-160,80){\framebox(10,10)}\put(-150,80){\framebox(10,10)}\put(-140,80){\framebox(20,10)}\put(-136,81){$\cdots$}\put(-120,80){\framebox(10,10)}
\put(-145,96){${k}\atop{{}}$}
\put(-80,80){\framebox(10,10)}\put(-70,80){\framebox(10,10)}\put(-60,80){\framebox(20,10)}\put(-57,81){$\cdots$}\put(-40,80){\framebox(10,10)}
\put(-65,96){${k-1}\atop{{}}$}\put(-80,70){\framebox(10,10)}
\put(0,80){\framebox(10,10)}\put(10,80){\framebox(10,10)}\put(20,80){\framebox(20,10)}\put(24,81){$\cdots$}\put(40,80){\framebox(10,10)}
\put(15,96){${k-2}\atop{{}}$}\put(0,70){\framebox(10,10)}\put(0,60){\framebox(10,10)}
\put(100,80){$\cdots$}
\end{picture}}\vs{-90pt}
\cl{\begin{picture}(30,160) \put(-132,80){$\cdots$}
\put(-82,60){{\scriptsize $n-2$}}
\put(-52,80){\framebox(10,10)}\put(-42,80){\framebox(10,10)}\put(-32,80){\framebox(20,10)}\put(-27,81){$\cdots$}\put(-12,80){\framebox(10,10)}
\put(-38,96){${k-n+3}\atop{{}}$}
\put(-52,70){\framebox(10,10)}\put(-52,50){\framebox(10,20)}\put(-48,56){$\vdots$}\put(-52,40){\framebox(10,10)}
\put(10,60){{\scriptsize $n-1$}}
\put(40,80){\framebox(10,10)}\put(50,80){\framebox(10,10)}\put(60,80){\framebox(20,10)}\put(63,81){$\cdots$}\put(80,80){\framebox(10,10)}
\put(40,70){\framebox(10,10)}\put(40,50){\framebox(10,20)}\put(44,54){$\vdots$}\put(40,40){\framebox(10,10)}
\put(52,96){${k-n+2}\atop{{}}$} \put(109,60){{\scriptsize $n$}}
\put(126,80){\framebox(10,10)}\put(136,80){\framebox(10,10)}\put(146,80){\framebox(20,10)}\put(149,81){$\cdots$}\put(166,80){\framebox(10,10)}
\put(126,70){\framebox(10,10)}\put(126,50){\framebox(10,20)}\put(129.5,56){$\vdots$}\put(126,40){\framebox(10,10)}
\put(138,96){${k-n+1}\atop{{}}$}\put(182,42){.}
\end{picture}}\vs{-40pt}

Denote by $Ch L_{\l^i_k}$ the character of the irreducible
$\gl$-representation with highest weight $\l^i$. Then
\begin{eqnarray*}
Ch L_{\l^i_k}=\frac{\mbox{$\sum\limits_{w\in
W}$}sign(w)e^{w((k-i+1)\e_1+\e_2+\cdots+\e_i+\widetilde{\rho})}}
{\mbox{$\prod\limits_{i<j}$}(e^{\e_i}-e^{\e_j})},
\end{eqnarray*}
where $\widetilde{\rho}=\rho+\frac{1}{2}(n-1)\sum_{i=1}^n\e_i$. The
formula is valid  when $k-i\ge 0$. When $i>k$, the right hand
vanishes identically.

\begin{lemm} Let $Ch_{n,k}$ be the expression obtained from $G_{n,k}$ by
replacing $L_i$ by $e^{\e_i}$ for all $i=1,2,\cdots,n$. Then
\begin{eqnarray}\label{0201m01}
\mbox{$\sum\limits_{i=1}^n$} (-1)^{i-1}q^{n-2i+1} Ch L_{\l^i_k}
=Ch_{n,k}, \quad \forall\,\,k=1,2,\cdots,
\end{eqnarray}
where $Ch L_{\l^i_k}=0$ if $k<i$.
\end{lemm}
{\bf Proof}\ \ Denote $W_{n-1}=\{w\in W\,|\,w(\e_1)=\e_1\}$.  Then
$W/W_{n-1}=\{[w] \mid w\in W\,|\,w(\e_1)=\e_j,\
\forall\,\,j\neq1\}$. Denote
$\widetilde{\rho}_{n-1}=\widetilde{\rho}-(n-1)\e_1$. In the case
when $k\geqslant n$, the left side of (\ref{0201m01}) can be rewritten as
follows:
\begin{eqnarray*}
&&\frac{q^{n-1}\mbox{$\sum\limits_{w\in W}$} sign(w)w\big(e^{k\e_1+\widetilde{\rho}}
\mbox{$\prod\limits_{i=2}^n$}(1-q^{-2}e^{-(\e_1-\e_i)})\big)}
{\mbox{$\prod\limits_{i<j}$}(e^{\e_i}-e^{\e_j})}\\
&&=q^{n-1}\mbox{$\sum\limits_{w\in W}$} w\frac{e^{k\e_1+\widetilde{\rho}}
\mbox{$\prod\limits_{i=2}^n$}(1-q^{-2}e^{-(\e_1-\e_i)})}
{\mbox{$\prod\limits_{i<j}$}(e^{\e_i}-e^{\e_j})}\\
&&=q^{n-1}\mbox{$\sum\limits_{[\sigma]\in W/W_{n-1}}$}
\mbox{$\sum\limits_{w'\in W_{n-1}}$}\sigma\big(
w'(\frac{e^{\widetilde{\rho}_{n-1}}}{\mbox{$\prod\limits_{1<i<j}$}(e^{\e_i}-e^{\e_j})})e^{(k+n-1)\e_1}
\mbox{$\prod\limits_{j=2}^n$}\frac{1-q^{-2}e^{-(\e_1-\e_j)}}{e^{\e_1}-e^{\e_j}}\big).
\end{eqnarray*}
According to the fact $Ch\,L_0=1$, we have the following identity
\begin{eqnarray*}
\mbox{$\sum\limits_{[w']\in W/W_{n-1}}$}
w'\frac{e^{\widetilde{\rho}_{n-1}}}{\mbox{$\prod\limits_{i<j}$}(e^{\e_i}-e^{\e_j})}=1.
\end{eqnarray*}

Hence the left side of (\ref{0201m01}) can be further simplified as follows:
\begin{eqnarray*}
&&q^{n-1}\mbox{$\sum\limits_{\sigma\in W/W_{n-1}}$}\sigma(e^{k+n-1)\e_1}
\mbox{$\prod\limits_{j=2}^n$}\frac{1-q^{-2}e^{-(\e_1-\e_j)}}{e^{\e_1}-e^{\e_j}})\\
&&=\mbox{$\sum\limits_{[\sigma]\in W/W_{n-1}}$}\sigma(e^{k\e_1}
\mbox{$\prod\limits_{j=2}^n$}\frac{qe^{\e_1}-q^{-1}e^{\e_j}}{e^{\e_1}-e^{\e_j}}),
\end{eqnarray*}
which is just the right side of (\ref{0201m01}), i.e. $G_{n,k}$
after  replacing $e^{\e_i}$ by $L_i$ for $i=1,2,\cdots,n$. Thus the
lemma in the case $k\geqslant n$ follows.

In the case $1\leqslant k<n$, the proof can be proceed similarly by
recalling the fact $ChL_{\l^i_k}=0$ for $i>k$. We
omit the details.
\QED

We have the following immediate consequence of the above lemma.

\begin{coro}
Set $\G_{k, i}=0$ if $i>k$. For $i\leqslant k$, let $\G_{k, i}$ be
obtained by replacing each $e^{\epsilon_i}$ in $Ch L_{\l_k^i}$ by
$L_i\in \Uq^0$.  Then
\begin{eqnarray*}
G_{n,k}&=&\mbox{$\sum\limits_{i=1}^n$} (-1)^{i-1} q^{n-2i+1} \G_{k, i}.
\end{eqnarray*}
\end{coro}

Note that the nonzero $\G_{k, i}$ are symmetric polynomials in
$L_1$, $L_2$, $\cdots$, $L_n$.

Let $\mathcal{G}_{n,k}$ denote the rational function in $q$ obtained
by replacing $e^{\e_i}$ by $q^{2(\e_i,\Lambda+\rho)}$ for  all
$i=1,2,\cdots,n$ on the left side of (\ref{0201m01}). Then we obtain
a new formula for the eigenvalues of the quantum Casimir operators
of $\Uq$ in the irreducible representation with highest weight $\L$.
\begin{coro}
\begin{eqnarray*}
\chi_\L(C_{n,k})=(q-q^{-1})^{-k}\sum_{l=0}^k(-1)^{k-l}
\left(\begin{array}{ll}
k\\l
\end{array}\right)(q^{1-n})^l\mathcal{G}_{n,l}.
\end{eqnarray*}
\end{coro}

The advantage of this formula is that every term in
$\mathcal{G}_{n,l}$ has a representation theoretical interpretation
in terms of the general linear Lie algebra.

\section{The main result}
With the preparations in the previous sections, we can now prove the
following theorem, which is the main result of the paper.
\begin{theo}\label{MT}
The center $\ZZ$ of $\Uq$ is generated by $c$ and the quantum
Casimir operators  $C_{n,1}$, $\cdots$, $C_{n,n}$.
\end{theo}

In order to prove the theorem, we need some basic results on
symmetric polynomials \cite{M}, which we recall here. The complete
homogeneous symmetric polynomial of degree $k$ in $n$ variables
$x_1$, $x_2$, $\cdots$, $x_n$, written $h_k$ for $k=0,1,2,\cdots$,
is the sum of all monomials of total degree $k$ in the variables.
Formally,
\begin{eqnarray*}
h_k(x_1,x_2,\cdots,x_n)=\mbox{$\sum\limits_{1\leqslant i_1\leqslant
i_2\leqslant \cdots\leqslant i_k\leqslant n}$}x_{i_1}x_{i_2}\cdots
x_{i_k}.
\end{eqnarray*}
It is well known that the set of complete homogeneous symmetric
polynomials
\[
h_1(x_1,x_2,\cdots,x_n), h_2(x_1,x_2,\cdots,x_n), \cdots,
h_n(x_1,x_2,\cdots,x_n)
\]
generate the ring of symmetric polynomials in the $n$ variables
$x_1$, $x_2$, $\cdots$, $x_n$.

\begin{proof}[Proof of Theorem \ref{MT}]
Denote $\K[L_1,L_2,\cdots,L_n]^W$ the algebra of symmetric
polynomials in the polynomial ring $\K[L_1,L_2,\cdots,L_n]$. Given
any element in $\UW$, we can always express it in terms of elements
of $\K[L_1,L_2,\cdots,L_n]^W$ and $c$ algebraically. Note that
$c^{-2}\in \K[L_1,L_2,\cdots,L_n]^W$. Therefore, in order to prove
Theorem \ref{MT}, it suffices to show that $G_{n,k}\,(k=1,\cdots,n)$
generate $\K[L_1,L_2,\cdots,L_n]^W$.

Note that $\G_{k, 1}$ in $G_{n,k}$ is a complete symmetric
polynomial in $L_1$, $\dots$, $L_n$. Thus $\G_{1, 1}$, $\G_{2, 1}$,
$\dots$, $\G_{n, 1}$ are a set of generators of the ring
$\K[L_1,L_2,\cdots,L_n]^W$ of symmetric polynomial.

Now $\G_{1, 1}$ is equal to $q^{1-n} G_{n,1}$, and we can easily
express $\G_{2, 1}$ in terms of $G_{n,1}$ and $G_{n,2}$. Inductively
we can show that $\G_{k, 1}$ can always be expressed in terms of
$G_{n,1}$, $G_{n,2}$, $\cdots$, $G_{n,k}$. Thus $G_{n,1}$,
$G_{n,2}$, $\cdots$, $G_{n,n}$ are also a set of generators of the
symmetric polynomial ring $\K[L_1,L_2,\cdots,L_n]^W$. Since the
elements of $\K[L_1,L_2,\cdots,L_n]^W$ and $c$ together generate
$\varphi(\ZZ)$, we complete the proof of the theorem.
\end{proof}

\vs{8pt}

{\bf Acknowledgements}\ \ This work is supported by the Australian
Research Council. Financial assistance from the National Science
Foundation of China (Grant No. 10926166) is also acknowledged.

\end{document}